\documentclass[onecolumn,10pt]{IEEEtran}

\usepackage{amsfonts,amsmath,amssymb,amsthm}
\usepackage{latexsym,amscd,caption}
\usepackage{amsbsy,algorithm,algorithmic}
\usepackage{graphicx,multirow,bm}
\usepackage[noadjust]{cite}
\usepackage{color}
\makeatletter
\newtheoremstyle{mythm}{3pt}{3pt}{}{16pt}{\bfseries}{:}{.5em}{}
\theoremstyle{mythm}
\newtheorem{theorem}{Theorem}
\newtheorem{definition}{Definition}
\newtheorem{remark}{Remark}

\newtheorem{corollary}{Corollary}
\newtheorem{lemma}{Lemma}

\newtheorem{construction}{Construction}

\DeclareMathOperator{\pr}{Pr}
\DeclareMathOperator{\dec}{Dec}
\DeclareMathOperator{\tr}{Tr}

\newcommand{\N}{\mathbb{N}}

\newcommand{\Z}{\mathbb{Z}}
\newcommand{\F}{\mathbb{F}}

\newcommand{\floornv}[1]{\left\lfloor #1 \right\rfloor}
\begin{document}

\title{Algebraic Manipulation Detection Codes via Highly Nonlinear Functions
\author{Minfeng Shao and Ying Miao}
\thanks{M. Shao and Y. Miao are with the Graduate School of Systems and Information Engineering,
University of Tsukuba, Tennodai 1-1-1, Tsukuba 305-8573, Japan
(e-mail: minfengshao@gmail.com, miao@sk.tsukuba.ac.jp).
This research is supported by JSPS Grant-in-Aid for Scientific Research (B) under Grant No. 18H01133.
}
 }
\date{}
\maketitle


\begin{abstract}
In this paper, we study the relationship between algebraic manipulation detection (AMD) codes
and highly nonlinear functions. As applications, on one hand, a generic construction for
systematic AMD codes is introduced based on highly nonlinear functions.
Systematic AMD codes with new parameters can be generated from known highly nonlinear functions.
Especially, several infinite classes of optimal systematic AMD codes,
some with asymptotically optimal tag size, can be constructed.
On the other hand, systematic AMD codes are used to construct highly nonlinear functions.
The known construction by Cramer \textit{et al.} \cite{CDFPW} for systematic AMD codes
turns out to be based on a special kind of functions with high nonlinearity.
\end{abstract}



\section{Introduction}

Algebraic manipulation detection (AMD) codes were first introduced by Cramer \textit{et al.} \cite{CDFPW}
to convert linear secret sharing schemes into robust secret sharing schemes
and build nearly optimal robust fuzzy extractors.
An AMD code can be viewed as an authentication code without a secret key.
Generally speaking, an AMD code consists of a probabilistic encoding map $E$ and
a determined decoding function $\dec$, where $E$ encodes a plaintext $s$ into a ciphertext $g\in_R G_s$
such that any tampering will be detected, except with a small constant error probability.
For AMD codes, we consider the attack model such that an adversary may manipulate
the valid message $g=E(s)$ by adding some offset $\Delta$ of his choice.
The attack model is divided into two sub-models by distinguishing two different settings:
the adversary has full knowledge of the source (the strong attack model) and
the adversary has no knowledge about the source (the weak attack model).
The main objective for an AMD code is, for a given tag size, to minimize the success probability
of an adversary such that $g+\Delta \in G_{s'}$, i.e., $\dec(g')=s'\ne\dec(g)=s$.
Here the tag size denotes the difference between the length of the plaintext and its corresponding ciphertext.

Upon to now, there are mainly two kinds of known constructions for AMD codes.
One is via algebraic methods. In \cite{CDFPW}, Cramer \textit{et al.} proposed a construction of AMD codes
with nearly optimal tag size based on polynomial evaluations. In \cite{KW}, Reed-Muller codes were
included to construct AMD codes.
Later in \cite{CPX}, linear codes such as BCH codes were included to generate AMD codes with small tag size.
The other one is via combinatorial method, which generates AMD codes by carefully designing
the combinatorial structures behind, i.e., the structure of image sets of the probabilistic encoding map $E$.
In \cite{CFP}, Cramer \textit{et al.} first introduced a kind of differential structures to construct AMD codes.
In \cite{PS}, Paterson and Stinson characterized optimal AMD codes with various types of
generalized external difference families (say, strong external difference families)
for different merits of optimality, respectively.
In \cite{HP} and \cite{SM}, for the weak attack model, combinatorial characterizations were given
for AMD codes via weighted external difference families.
In the past few years, many efforts have been devoted to construct AMD codes and their corresponding
generalized external difference families (see \cite{JL,LNC,MS,PS,WYF,WYFF}, and the references therein).

In cryptography, another interesting topic closely related to authentication codes
are functions with high nonlinearity \cite{CD2004}.
Highly nonlinear functions such as bent functions received much attention not only for their applications
in cryptography (for instances, combining keystream generators for stream ciphers \cite{CDR1998},
$S$-boxes for block ciphers \cite{CD2007,N}, functions with optimal algebraic immunity \cite{TCT}, 
and authentication codes \cite{CD2006,DN2004}),
sequences design \cite{DHM,OSW} and coding theory \cite{CCD,WLZ}, but also for
their close relationship with combinatorial designs \cite{Dillon,R}.

In this paper, we study the relationship between AMD codes and highly nonlinear functions.
On one hand, we propose a generic construction of AMD codes via functions with high nonlinearity.
By choosing special highly nonlinear functions such as perfect nonlinear functions,
a few infinite classes of AMD codes with new parameters can be generated for both weak and strong attack models.
For the weak attack model, $R$-optimal AMD codes have asymptotically optimal tag size can be constructed.
For the strong attack model, some AMD codes generated by our construction are proved to have
the minimum possible probability of successful tampering.
On the other hand, we try to construct highly nonlinear functions from known AMD codes.
Based on a subclass of AMD codes with more strict assumptions, highly nonlinear functions can be generated,
where their nonlinearities are determined by the parameters of the corresponding AMD codes.
Especially, we prove that the known construction in \cite[Theorem 2]{CDFPW} can also be explained
by highly nonlinear functions.

The remainder of this paper is organized as follows.
In Section \ref{sec-preliminary}, we introduce some preliminaries about AMD codes.
In Section \ref{sec-construction-weak}, we construct systematic AMD codes
via highly nonlinear functions for the weak attack model, whereas
Section \ref{sec-construction-strong} is devoted to construct systematic
AMD codes under the strong attack model.
In Section \ref{sec-HNF}, highly nonlinear functions are constructed based on known systematic AMD codes.
Conclusion is drawn in Section \ref{sec-conclusion}.


\section{Preliminaries}
\label{sec-preliminary}

In this section, we recap some notation, definitions and results about AMD codes.

\begin{definition}
\label{def_AMD}
Let $S$ be a set of plaintext messages with size $m$ termed the \textit{source space},
and $G$ be the \textit{encoded message space}, which is usually an Abelian group of order $n$.
Consider a pair of a probabilistic encoding map $E: S \rightarrow G$ and
a deterministic decoding function $\dec: G \rightarrow S \cup \{\bot\}$ such that
$\dec(E(s)) = s$ with probability $1$ for any $s \in S$.
Let $G_s$ be the set of \textit{valid} encodings of $s \in S$, i.e.,
$G_{s}\triangleq \{g\in G: \dec(g)=s\}$.
\begin{itemize}
\item[(1)] The pair $(E,\dec)$ is called a \textit{strong} $(m,n,\rho)$
\textit{algebraic manipulation detection} (AMD) \textit{code} if for any $s\in S$,
$\Delta \in G \setminus \{0\}$, the probability of $\dec(E(s)+\Delta)\not \in \{s,\bot\}$
is at most $\rho$, i.e., $\pr(\dec(E(s)+\Delta) \not \in \{s,\bot\}) \le \rho\leq 1$.
\item[(2)] The pair $(E,\dec)$ is called a \textit{weak} $(m,n,\rho)$-AMD \textit{code} if for any
$\Delta \in G \setminus \{0\}$ and any random $s\in_R S$ rather than an arbitrary one, the probability
$\sum_{s\in S}\pr(s)\sum_{g \in G_s}\pr(E(s)=g)\pr(\dec(g+\Delta)) \not \in \{s,\bot\}) \le \rho$.
\item[(3)] An AMD code $(E,\dec)$, whether strong or weak, is called \textit{systematic}
if $S=A_1$ is an Abelian group, $G$ is an Abelian group $A_1 \times A_2 \times B$, and the encoding has the form
\begin{equation}\label{eqn_def_AMD}
E: A_1 \rightarrow A_1\times A_2\times B\text{ \ with } E(s) = (s,x,f(s,x))
\end{equation}
for some function $f: A_1 \times A_2 \rightarrow B$ and $x \in_R A_2$.
For a systematic AMD code, the decoding function is naturally given by
\begin{equation}\label{eqn_decoding}
\dec(s,x,t) = \begin{cases}s, &\text{if } t=f(s,x),\\
\bot, &\text{otherwise}.\\
\end{cases}
\end{equation}
\end{itemize}
\end{definition}

Note that by randomly adding $(\delta_1, \delta_2, \delta_3) \in A_1 \times A_2 \times A_3$
to $(s,x,f(s,x))$, we can make $(s,x,f(s,x))$ unreadable to the adversary.

Throughout this paper, we fix the following notation for AMD codes.
\begin{itemize}
\item An $(m,n,\rho)$-AMD code is said to have \textit{equiprobable sources} if
$\pr(s) = \frac{1}{m}$ for any $s \in S$.
\item An $(m,n,\rho)$-AMD code is said to be \textit{equiprobable encoding} if
$\pr(E(s)=g) = \frac{1}{|G_s|}$ for any $s \in S$ and $g \in G_s$.
\item An $(m,n,\rho)$-AMD code is said to be \textit{uniform} if
$|G_s|$ is constant for any $s \in S$.
\item A uniform $(m,n,\rho)$-AMD code with $|G_s|=t$ for $s\in S$ is said to
be \textit{$t$-regular} if it has equiprobable sources and equiprobable encoding.
\end{itemize}

For a systematic AMD code, if $\pr(E(s)=(s,x=x_0,f(s,x_0))\neq 0$ for any $s \in A_1$ and $x \in_R A_2$
that is for any $x_0\in A_2$, $\pr(x=x_0)>0$,
then it is $|A_2|$-uniform.
Thus, an equiprobable encoding systematic AMD code with equiprobable sources is $|A_2|$-regular.

\begin{definition}[\cite{CDFPW}]
The \textit{tag size} of an $(m,n,\rho)$-AMD code is $\varpi=\log|G|-\log|S|=\log n-\log m$.
\end{definition}

For the convenience of theoretic analysis, for any $u,k\in \N$, define \textit{effective tag size} as
$\varpi^*(k,u)=\min\{\log|G|\}-u$, where the minimum is over all $(|S|,|G|,\rho)$-AMD codes
such that $|S|\geq 2^{u}$ and $\rho\leq 2^{-k}$.
In \cite{CDFPW}, Cramer \textit{et al.} derived a lower bound for $\varpi^*(k,u)$ as follows.

\begin{lemma}[\cite{CDFPW}]
\label{lemma_tag}
For any $u,k\in \N$, the effective tag size is lower bounded by
\begin{equation*}
\varpi^*(k,u)\geq 2k-2^{-u+1}\geq 2k-1
\end{equation*}
for strong AMD codes, and
\begin{equation*}
\varpi^*(k,u)\geq k-2^{-u+1}\geq k-1
\end{equation*}
for weak AMD codes, respectively.
\end{lemma}

Besides the above bound for the effective tag size of an AMD code,
the following theoretic bounds on the parameters are also known.

\begin{lemma}[\cite{PS}]
\label{lemma_bound_para}
For any weak $(m,n,\rho)$-AMD code, we have
\begin{equation*}
\rho \geq \frac{a(m-1)}{m(n-1)},
\end{equation*}
where $a=\sum_{s\in S}|G_s|$.
\end{lemma}

Especially, for $t$-regular weak AMD codes, the probability of successful tampering is lower bounded as follows.

\begin{lemma}[\cite{PS,SM}]
\label{lemma_bound_uniform}
For any $t$-regular weak $(m,n,\rho)$-AMD codes, we have
\begin{equation}\label{eqn_Bound_lambda}
\rho\geq \left\lceil\frac{t^2m(m-1)}{n-1}\right\rceil\frac{1}{tm}.
\end{equation}
\end{lemma}

\begin{definition}[\cite{PS,SM}]
\label{def_R_op_PS}
A weak AMD code is $R$-\textit{optimal} if its parameters meet the bound in Lemma \ref{lemma_bound_para}
with equality (or Lemma \ref{lemma_bound_uniform} for the $t$-regular case).
Here, ``$R$" indicates that random choosing $\Delta$ is an optimal strategy for the adversary.
\end{definition}

\begin{lemma}[\cite{PS}]
\label{lemma_bound_G}
For any strong $(m,n,\rho)$-AMD code, we have
\begin{equation}\label{eqn_bound_m1021}
{\rho}_s \geq \frac{1}{|G_s|}
\end{equation}
for any source $s \in S$, where $\rho_s$ is the probability of successful tampering
given the source $s \in S$ for a random chosen $\Delta$.
\end{lemma}

\begin{definition}[\cite{PS}]
\label{def_G_op_PS}
A strong AMD code is $G$-\textit{optimal} if its parameters meet
the bound in Lemma \ref{lemma_bound_G} with equality.
Here, ``$G$" indicates that guessing the most likely encoding is an optimal strategy for the adversary.
\end{definition}


\section{Weak algebraic manipulation detection codes from highly nonlinear functions}
\label{sec-construction-weak}

In this section, we propose a construction for systematic weak AMD codes via highly nonlinear functions.
Systematic weak AMD codes with asymptotically optimal effective tag size are constructed in
Corollaries \ref{coro_LF} and \ref{cor_optimalm}, and $R$-optimal systematic weak AMD codes
are constructed in Corollaries \ref{coro_MM} and \ref{cor_weak_Dillon}.

First of all, we recall some necessary definitions about nonlinearity of functions.

Let $(A,+)$ and $(B,+)$ be two Abelian groups with order $n$ and $m$, respectively.
Let $f$ be a function from $A$ to $B$.
One robust way to measure the nonlinearity of a function $f$ from $A$ to $B$ is
to use the derivatives $D_a(f(x))=f(x+a)-f(x)$ for $a\in A$,
which is closely related to differential cryptanalysis \cite{BS,N}.

\begin{definition}[\cite{N}]
\label{def_nonliearity}
The \textit{nonlinearity} $P_f$ of a function $f$ from $A$ to $B$ is defined as
\begin{equation}\label{eqn_def_Pf}
P_f\triangleq\max_{a\in A\setminus \{0\}}\max_{b\in B} \pr\left(D_a(f(x))=b\right)=
\max_{a\in A\setminus\{0\}}\max_{b\in B}\frac{|\{x\in A: ~ D_a(f(x))=b\}|}{|A|},
\end{equation}
where $\pr(D_a(f(x))=b)$ denotes the probability of the occurrence of the event $D_a(f(x))=b$.
\end{definition}

\begin{remark}
The \textit{Hamming distance} between two functions $f$ and $g$
from $A$ to $B$ is defined to be $d(f,g) = |\{x \in A: f(x) \neq g(x)\}|$.
A function $f$ is \textit{linear} if and only if $f(x+y) = f(x) + f(y)$ for all $x,y \in A$.
A function $g$ is \textit{affine} if and only if $g = f + b$, where $f$ is linear and $b$ is a constant.
An alternative method of measuring the nonlinearity of a function $f: A \rightarrow B$ is given
by the minimum Hamming distance between $f$ and all possible affine functions from $A$ to $B$ \cite{OSW}.
This measure of nonlinearity is closely related to linear cryptanalysis \cite{M}.
For the relationship between these two definitions of nonlinearity,
the reader is referred to \cite{CD,CD2004}, for instances.
In this paper, the former definition of nonlinearity is used.
\end{remark}

It is easy to check (see, for example, \cite{CD2004}) that $P_f=1$ if $f$ is a linear function
from $A$ to $B$, and $P_f\geq \frac{1}{|B|}$ for any function $f$ from $A$ to $B$.
The smaller the value of $P_f$, the higher the corresponding nonlinearity of $f$.

\begin{definition}[\cite{N}]
A function $f$ from $A$ to $B$ is said to have \textit{perfect nonlinearity} if $P_f=\frac{1}{|B|}$.
\end{definition}

\begin{construction}
\label{cons_nonlinear_weak}
Let $f$ be a function from $A=A_1\times A_2$ to $B$ and let $S=A_1$ be a subgroup of $A$.
Define a probabilistic encoding map $E_f: A_1 \rightarrow G=A\times B=A_1\times A_2 \times B$ as
\begin{equation}\label{eqn_consB}
E_f(S_1) = (S_1,S_2,f(S_1,S_2)),
\end{equation}
where $S_2 \in_R A_2$.
\end{construction}

By the probabilistic encoding map $E_f$ and the corresponding deterministic decoding function
given by \eqref{eqn_decoding}, we can define a systematic AMD code $(E_f,\dec)$ from $A_1$ to
$G=A\times B=A_1\times A_2\times B$.
Note that a possible successful tampering should satisfy $\Delta\in (A_1\setminus \{0\})\times A_2\times B$.
However, for the nonlinearity, we should consider all possible
$\Delta\in A_1\times A_2\times B \setminus \{(0,0,0)\}$.
Thus, to the convenience of analysis, we introduce the partial nonlinearity of a function, which
only considers the case $\Delta\in (A_1\setminus\{0\})\times A_2 \times B$.

\begin{definition}
\label{def_partial_nonlinearity}
The \textit{partial nonlinearity} $\Psi_f(A_1)$ of a function $f$ from $A=A_1\times A_2$ to $B$ is defined as
\begin{equation}\label{eqn_def_Omega}
\Psi_f(A_1)\triangleq\max_{a_1\in A_1\setminus\{0\}}\max_{a_2\in A_2}\max_{b\in B} \pr\left(D_{(a_1,a_2)}(f(x))=b\right)
=\max_{a_1\in A_1\setminus\{0\}}\max_{a_2\in A_2}\max_{b\in B}\frac{|\{x\in A: ~ D_{(a_1,a_2)}(f(x))=b\}|}{|A|},
\end{equation}
where $A_1$ is a subgroup of $A$ and $\pr\left(D_{(a_1,a_2)}(f(x))=b\right)$ denotes the probability
of the occurrence of the event $D_{(a_1,a_2)}(f(x))=b$.
\end{definition}

\begin{remark}
\label{remark_partial_nonliearity}
By Definitions \ref{def_nonliearity} and \ref{def_partial_nonlinearity}, we have
$P_f\geq \Psi_f(A_1)$ for any subgroup $A_1$ of $A$.
\end{remark}

The parameters of the constructed AMD code have the following relationship with the nonlinearity of $f$.

\begin{theorem}
\label{thm_nonlinear_weak}
If the function $f$ from $A=A_1\times A_2$ to $B$ with partial nonlinearity $\Psi_f(A_1)$
has equiprobable sources and $E_f$ is equiprobable encoding, then the systematic
weak AMD code $(E_f, \dec)$ generated by Construction \ref{cons_nonlinear_weak} has parameters
$(n_1,n_1n_2m,\Psi_f(A_1)\leq P_f)$, where $|A_1|=n_1$, $|A_2|=n_2$, $|B|=m$ and $P_f$
denotes the nonlinearity of $f$.
\end{theorem}

\begin{proof}
By Construction \ref{cons_nonlinear_weak}, we only need to prove that the probability of
successful tampering is upper bounded by $\Psi_f(A_1)$.
For any $\Delta=(a_1,a_2,b)\in G = A_1 \times A_2 \times B$ with
$a_1\in A_1\setminus \{0\}$, $a_2\in A_2$ and $b\in B$,
\begin{eqnarray}\label{miao1}
\nonumber
& & \sum_{S_1\in A_1}\pr(S'=S_1)\sum_{S_2\in A_2}\pr(E_f(S_1)=(S_1,S_2,f(S_1,S_2))\pr({\dec}((S_1,S_2,f(S_1,S_2))+\Delta)\not\in \{S_1,\bot\}) \\
& = & \sum_{S_1\in A_1}\pr(S'=S_1)\sum_{S_2\in A_2}\pr(S^*=S_2)\pr({\dec}((S_1,S_2,f(S_1,S_2))+\Delta)\not\in \{S_1,\bot\}) \\ \nonumber
& = & \sum_{S_1\in A_1}\frac{1}{|A_1|}\sum_{S_2\in A_2}\frac{1}{|A_2|}\pr(f(S_1+a_1,S_2+a_2)=f(S_1,S_2)+b) \\ \nonumber
& = & \frac{1}{|A_1||A_2|}\sum_{S_1\in A_1}\sum_{S_2\in A_2}\pr(f(S_1+a_1,S_2+a_2)=f(S_1,S_2)+b) \\ \nonumber
& = & \frac{1}{|A_1||A_2|}|\{(S',S^*)\in A: ~ f(S'+a_1,S^*+a_2)-f(S',S^*)=b\}| \\ \nonumber
& \leq & \Psi_f(A_1)\leq P_f.
\end{eqnarray}
According to Definitions \ref{def_AMD} and \ref{def_partial_nonlinearity},
we know that the systematic weak AMD code $(E_f, \dec)$ generated by Construction \ref{cons_nonlinear_weak}
has parameters $(n_1,n_1n_2m,\Psi_f(A_1)\leq P_f)$, which completes the proof.
\end{proof}

In what follows, we list some well-known highly nonlinear functions and their
corresponding systematic AMD codes as applications of Construction \ref{cons_nonlinear_weak}.


\subsection{Linear functions}

One simple but useful way to obtain functions with high nonlinearity is to use linear functions
from $(\F_{q^r},+)$ to $(\F_{q},+)$ as functions from $(\F^*_{q^r},\times)\cong (\Z_{q^r-1},+)$ to $(\F_q,+)$.

\begin{lemma}[\cite{CD2004}]
\label{lemma_linear_func}
Any nonzero linear function $L$ from $(\F_{q^r},+)$ to $(\F_{q},+)$ is a function
from $(\F^*_{q^r},\times)$ to $(\F_q,+)$ with nonlinearity $P_f=\frac{1}{q}+\frac{1}{q(q^r-1)}$.
\end{lemma}

Applying the highly nonlinear functions in Lemma \ref{lemma_linear_func}, the following corollary
follows directly from Construction \ref{cons_nonlinear_weak} and Theorem \ref{thm_nonlinear_weak}.

\begin{corollary}
\label{coro_LF}
Let $A=(\Z_{q^r-1},+)$ and $B=(\F_q,+)$. Further let $A_1=(\Z_{m_1},+)$ and $A_2=(\Z_{m_2},+)$
be two subgroups of $A$ with order $m_1$ and $m_2= \frac{q^r-1}{m_1}$, respectively.
If $\gcd(m_1,m_2)=1$, then $A \cong A_1\times A_2$.
Define the probabilistic encoding map $E_L$ from $A_1$ to $G=A_1\times A_2\times B$ as
\begin{equation}\label{m.1017.1}
E_L(s_1)= (s_1,s_2,L(\Phi(s_1,s_2))),
\end{equation}
where $s_2\in_R A_2$, $\Phi$ is an isomorphism from $\Z_{q^r-1}$ to $(\F^*_{q^r},\times)$,
and $L(x)$ is a nonzero linear function from $(\F^*_{q^r},\times)$ to $(\F_q,+)$.
If the systematic weak AMD code $(E_L, \dec)$ given by \eqref{m.1017.1} and \eqref{eqn_decoding}
has equiprobable sources and $E_L$ is equiprobable encoding,
then it is an $m_2$-regular $(m_1,(q^{r}-1)q,\frac{1}{q}+\frac{1}{q(q^r-1)})$-AMD code.
\end{corollary}

\begin{corollary}
\label{cor_optimalm}
Let $r \in \N$ and $m_2$ be a factor of $q^r-1$. Further let $m_1=\frac{q^r-1}{m_2}$,
$u=\lfloor\log m_1 \rfloor$, and $k=\lfloor\log \frac{q^{r}-1}{q^{r-1}}\rfloor$.
If $\gcd(m_1,m_2)=1,$ then the effective tag size $\varpi^*(k,u)$ for weak AMD codes satisfies
$$k-1 \leq \varpi^*(k,u) <k +1+ \log \frac{m_2q^r}{q^r-1}.$$
The systematic weak AMD code in Corollary \ref{coro_LF} has an asymptotically optimal
effective tag size with respect to the bound in Lemma \ref{lemma_tag},
i.e., $\lim_{k\rightarrow \infty}\frac{\log|G|-\log|A_1|}{k-1}=1$.
\end{corollary}

\begin{IEEEproof}
By Corollary \ref{coro_LF}, there exists a systematic weak AMD code with
$|A_1|=\frac{q^r-1}{m_2}\geq 2^u$, $\rho=\frac{q^{r-1}}{q^{r}-1}\leq 2^{-k}$, and the tag size
$$\varpi=\log|G|-\log|A_1|=\log(m_2q)=\log m_2+\log q.$$
Note that
\begin{equation*}
\varpi-k=\log m_2 +\log q - \floornv{\log \frac{q^{r}-1}{q^{r-1}}} < 1+\log m_2 + \log \frac{q^{r}}{q^r-1}.
\end{equation*}
The first conclusion then follows from the fact that for any $k,u \in \N$,
we have $k-1 \le \varpi^*(k,u) \le \varpi$.
The second conclusion can be derived by the fact that
\begin{equation*}
\lim_{k\rightarrow \infty}\frac{\log|G|-\log|A_1|}{k-1}=\lim_{q\rightarrow \infty}\frac{1+\log m_2+\log \frac{q^r}{q^r-1}+k}{k-1}=1
\end{equation*}
by noting that $m_2 \in \N$ is a constant.
\end{IEEEproof}


\subsection{Maiorana-McFarland's class of functions}

Let $r \in \N$ and $q$ be a prime power. Define a function
$f: (\F^{2r}_q,+) \rightarrow (\F_q,+)$ as
\begin{equation}\label{eqn_MM}
f(x_1,x_2,\ldots,x_{2r})=\sum_{1\le i\le r}x_ix_{i+r}.
\end{equation}

\begin{lemma}[\cite{M2016}]
\label{lemma_Malorana}
The function $f(x_1,x_2,\ldots,x_{2r})$ defined by \eqref{eqn_MM} has perfect nonlinearity $P_f=\frac{1}{q}$.
\end{lemma}

\begin{corollary}
\label{coro_MM}
Let $A_1 = (\F_{q^{2r-1}},+)$ and $A_2=B=(\F_q,+)$,
where we regard an element of $\F_{q^{2r-1}}$ as a vector in $\F^{2r-1}_q$.
Define the probabilistic encoding map $E_f$ from $A_1$ to $G=A_1\times A_2\times B$ as
\begin{equation}\label{m.1015.1}
E_f(S_1) = (S_1,s_2,f(S_1,s_2))=\left(x_1,x_2,\ldots,x_{2r-1},s_2,x_rs_2+\sum_{1\le i\le r-1}x_ix_{i+r}\right),
\end{equation}
where $S_1=(x_1,x_2,\ldots,x_{2r-1}) \in A_1$, $s_2\in_R A_2$, and $f$ is defined by \eqref{eqn_MM}.
If the systematic weak AMD code $(E_f, \dec)$ given by \eqref{m.1015.1} and \eqref{eqn_decoding}
has equiprobable sources and $E_f$ is equiprobable encoding,
then it is an $R$-optimal $q$-regular $(q^{2r-1},q^{2r+1},\frac{1}{q})$-AMD code
with respect to the bound in Lemma \ref{lemma_bound_uniform}.
\end{corollary}

\begin{proof}
The statement that the constructed AMD code $(E_f, \dec)$ has parameters $(q^{2r-1},q^{2r+1},\frac{1}{q})$
directly follows from Theorem \ref{thm_nonlinear_weak}, Lemma \ref{lemma_Malorana},
and the fact that it is $q$-regular. According to Lemma \ref{lemma_bound_uniform}, we should have
$$\rho\geq \left\lceil\frac{q^2q^{2r-1}(q^{2r-1}-1)}{q^{2r+1}-1}\right\rceil\frac{1}{q^{2r}}=\frac{1}{q},$$
which means that the constructed AMD code is $R$-optimal.
\end{proof}

\begin{corollary}
\label{cor_optimal}
For any $k,u \in \N$, the effective tag size $\varpi^*(k,u)$ for weak AMD codes is bounded as follows:
$$k-1 \leq \varpi^*(k,u) \le 2k.$$
\end{corollary}

\begin{IEEEproof}
For any given $k$ and $u$, choose $q=2^k$ and $r$ to be the smallest positive integer such that $u \le k(2r-1)$.
According to Corollary \ref{coro_MM}, there exists a systematic weak AMD code with
$|A_1|=q^{2r-1}\geq 2^u$, $\rho=\frac{1}{q}\leq 2^{-k}$, and the tag size
$\varpi=\log|G|-\log|A_1|=2\log q = 2k$.
Then the claim follows from the fact that
$k-1\leq \varpi^*(k,u)\leq \varpi$.
\end{IEEEproof}


\subsection{Dillon's class of functions}

In this subsection, we recall the well-known Dillon's class of functions with perfect nonlinearity.
A function $g: A \rightarrow B$ is \textit{balanced} if the size of $g^{-1}(b)$ is the same
for every $b \in B$, which is $|A|/|B|$.
It is known (see, for example, \cite{CD2004}) that $g$ has perfect nonlinearity if and only if
for every $a \in A \setminus \{0\}$, the derivative $D_a(g(x))$ is balanced,
and this is possible only when $|B|$ divides $|A|$.

\begin{lemma}[\cite{Dillon}]
\label{lemma_Dillon}
For any $r \in \N$, let $\F^{r}_{q}$ be identified with the finite field $\F_{q^{r}}$
and let $g$ be any balanced function from $\F_{q^r}$ to $\F_q$.
Then the function $f: (\F_{q^{2r}},+) \rightarrow (\F_{q},+)$ defined by
$$f(x,y)=g(xy^{q^{r}-2}),\,\, \ \ x,y\in \F_{q^{r}}$$
has perfect nonlinearity $P_f=\frac{1}{q}$.
\end{lemma}

\begin{corollary}
\label{cor_weak_Dillon}
Let $A_1=(\F_{q^{2r-1}},+)$ and $A_2=B=(\F_q,+)$, where we regard an element of $\F_{q^{2r-1}}$
as a vector in $\F^{2r-1}_q$. Let $g: \F_{q^r} \rightarrow \F_q$ be a balanced function.
Define the probabilistic encoding map $E_f$ from $A_1$ to $G=A_1\times A_2\times B$ as
\begin{equation}\label{m.1017.2}
E_f(S_1) = (S_1,y_r,f(X,Y))=\left(x_1,x_2,\ldots,x_r,y_1,y_2,\ldots,y_{r-1},y_r,g(XY^{q^r-2})\right),
\end{equation}
where $S_1=(x_1,x_2,\ldots,x_r,y_1,y_2,\ldots,y_{r-1}) \in A_1$, $y_r\in_R A_2$,
$X=(x_1,x_2,\ldots,x_r)$, and $Y=(y_1,y_2,\ldots,y_{r})$.
If the systematic weak AMD code $(E_f, \dec)$ given by \eqref{m.1017.2} and \eqref{eqn_decoding}
has equiprobable sources and $E_f$ is equiprobable encoding,
then it is an $R$-optimal $q$-regular $(q^{2r-1},q^{2r+1},\frac{1}{q})$-AMD code
with respect to the bound in Lemma \ref{lemma_bound_uniform}.
\end{corollary}

The proof of Corollary \ref{cor_weak_Dillon} is similar to that of Corollary \ref{coro_MM} so we omit it here.
Note that although the parameters of the systematic weak AMD codes constructed in Corollaries
\ref{coro_MM} and \ref{cor_weak_Dillon} are the same, their probabilistic encoding maps are different.


\section{Strong algebraic manipulation detection codes from highly nonlinear functions}
\label{sec-construction-strong}

In this section, we consider the systematic strong AMD codes generated
by Construction \ref{cons_nonlinear_weak} via highly nonlinear functions.
We first analyze the relationship between the nonlinearity of the function $f$ and
the probability of successful tampering in Theorem \ref{thm_nonlinear}.
By choosing some special functions, we construct systematic strong AMD codes
as examples in Corollaries \ref{coro_MM1} and \ref{cor_strong_Dillon}.

By the probabilistic encoding map $E_f$ given by \eqref{eqn_consB} and
the corresponding decoding function given by \eqref{eqn_decoding},
we can define a systematic AMD code $(E_f,\dec)$ from $A_1$ to $G=A\times B=A_1\times A_2\times B$.
In what follows, we first analyze the relationship between parameters of the strong AMD code
generated by Construction \ref{cons_nonlinear_weak} and the nonlinearity of $f$.

\begin{theorem}
\label{thm_nonlinear}
Let $f$ be a function from $A=A_1\times A_2$ to $B$ with nonlinearity $P_f$ and partial nonlinearity $\Psi_f(A_1)$,
where $|A_1|=n_1, |A_2|=n_2$, and $|B|=m$.
For the equiprobable encoding case, the systematic strong AMD code $(E_f, \dec)$
generated by Construction \ref{cons_nonlinear_weak} has parameters $(n_1,n_1n_2m,\rho)$
if and only if for any given $S_1\in A_1$,
\begin{equation}\label{eqn_local_PN}
\frac{|\{S^*\in A_2~:~f(S_1+a_1,S^*+a_2)=f(S_1,S^*)+b\}|}{|A_2|}\leq \rho
\end{equation}
holds for any $\Delta=(a_1,a_2,b)\in (A_1 \setminus \{0\})\times A_2\times B$.
The parameter $\rho$ satisfies that $\rho\geq \Psi_f(A_1)\geq \frac{1}{|B|}$.
Furthermore, if $f$ is a perfect nonlinear function, then we have $\rho\geq P_f$.
\end{theorem}

\begin{proof}
By Construction \ref{cons_nonlinear_weak} and Definition \ref{def_AMD},
the AMD code $(E_f, \dec)$ generated by Construction \ref{cons_nonlinear_weak} has
parameters $(n_1,n_1n_2m,\rho)$ if and only if for any $S_1\in A_1$,
\begin{equation}\label{eqn_AMD_rho}
\pr({\dec}(E_f(S_1)+\Delta)\not\in  \{S_1,\bot\})\leq \rho
\end{equation}
holds for any $\Delta=(a_1,a_2,b)\in G$ with $a\in A_1 \setminus \{0\}$, $a_2\in A_2$, and $b\in B$.
But
\begin{eqnarray*}
\pr({\dec}(E_f(S_1)+\Delta)\not\in \{S_1,\bot\})
& = & \sum_{S^*\in A_2}\pr(S_2=S^*)\pr(f(S_1+a_1,S^*+a_2)=f(S_1,S^*)+b)\\
& = & \sum_{S^*\in A_2}\frac{1}{|A_2|}\pr(f(S_1+a_1,S^*+a_2)=f(S_1,S^*)+b)\\
& = & \frac{|\{S^*\in A_2: \ f(S_1+a_1,S^*+a_2)=f(S_1,S^*)+b\}|}{|A_2|},
\end{eqnarray*}
where the second equality follows from the fact that $E_f$ is equiprobable encoding.
Therefore, \eqref{eqn_AMD_rho} is equivalent with \eqref{eqn_local_PN}.

Now we prove $\rho\geq \Psi_f(A_1)$.
Clearly, there exists a fixed $\Delta=(a_1,a_2,b)\in (A_1\setminus \{0\})\times
A_2\times B$ such that $\Psi_f(A_1) = \pr(D_{(a_1,a_2)}(f(S_1,S_2))=b)$. Then
\begin{eqnarray*}
\Psi_f(A_1) & = & \pr(D_{(a_1,a_2)}(f(S_1,S_2)=b) \\
& = & \frac{\sum\limits_{S_1\in A_1}|\{S^*\in A_2: \ f(S_1+a_1,S^*+a_2)=f(S_1,S^*)+b\}|}{|A_1||A_2|}\\
& \leq & \frac{\sum\limits_{S_1\in A_1}\rho}{|A_1|}\\
& = & \rho.
\end{eqnarray*}

For any $a_1\in A_1 \setminus \{0\}$ and $a_2\in A_2$,
\begin{eqnarray*}
& & \max_{b\in B}\frac{|\{(S',S^*)\in A_1 \times A_2: D_{(a_1,a_2)}(f(S',S^*))=b\}|}{|A_1||A_2|} \\
& \geq & \frac{\sum\limits_{b\in B}\frac{|\{(S',S^*)\in A_1 \times A_2: \ D_{(a_1,a_2)}(f(S',S^*))=b\}|}{|A_1||A_2|}}{|B|}\\
& = &\frac{1}{|B|}
\end{eqnarray*}
implies that $\Psi_f(A_1) \geq \frac{1}{|B|}$.

At last, if $f$ is a perfect nonlinear function, then we have
$\frac{1}{|B|}=P_f\geq \Psi_f(A_1)\geq \frac{1}{|B|}$ according to Remark \ref{remark_partial_nonliearity}.
Thus, we have $\rho\geq \Psi_f(A_1)=P_f=\frac{1}{|B|}$.
\end{proof}

In what follows, we list a few systematic strong AMD codes with $\rho=\frac{1}{|B|}$.
Especially, we include the classes of Maiorana-McFarland functions and Dillon functions to construct such AMD codes.

Based on Theorem \ref{thm_nonlinear} and Lemma \ref{cor_optimal}, we have the following corollary. Herein
we highlight that
this explicit construction was first introduced in \cite{KW}. We recall it as an application of
our generic construction and the prove is only for completeness.

\begin{corollary}[\cite{KW}]
\label{coro_MM1}
Let $A_1=A_2=\F_{q^r}$ and $B=\F_q$, where we regard an element of $\F_{q^r}$ as a vector in $\F^{r}_q$.
Define the probabilistic encoding map $E_f$ from $A_1$ to $G=A_1\times A_2\times B$ as
$$E_f(S_1) = (S_1,S_2,f(S_1,S_2))=\left(x_1,x_2,\ldots,x_r,x_{r+1},\ldots,x_{2r},\sum_{1\leq i\leq r}x_ix_{i+r}\right),$$
where $S_1=(x_1,x_2,\ldots,x_{r}) \in A_1$, $S_2=(x_{r+1},x_{r+2},\ldots,x_{2r})\in_R A_2$,
and $f$ is defined by \eqref{eqn_MM}. Then, for the equiprobable encoding case,
the systematic strong AMD code given by $E_f$ has parameters $(q^{r},q^{2r+1},\frac{1}{q})$,
where $\rho=\frac{1}{q}$ is minimum with respect to Theorem \ref{thm_nonlinear}.
Especially, when $r=1$, the $q$-regular AMD code is $G$-optimal
with respect to the bound in Lemma \ref{eqn_bound_m1021}.
\end{corollary}

\begin{IEEEproof}
We first prove $\rho=P_f=\frac{1}{q}$. By Theorem \ref{thm_nonlinear}, we only need
to prove that for any $S_1\in \F_{q^r}$, $a_1\in \F_{q^r} \setminus \{0\}$, $a_2\in \F_{q^r}$, and $b\in \F_q$,
\begin{equation}\label{eqn_MC}
\frac{|\{S_2\in \F_{q^r}: \ f(S_1+a_1,S_2+a_2)=f(S_1,S_2)+b\}|}{q^r}\leq \frac{1}{q},
\end{equation}
i.e., $f(S_1+a_1,S_2+a_2)=f(S_1,S_2)+b$ has at most $q^{r-1}$ solutions for $S_2\in \F_{q^r}$.
Since $a_1\ne 0$, without loss of generality, we may assume $a_1=(a_{11},a_{12},\ldots,a_{1r})$
with $a_{1r}\ne 0$. Note that
\begin{equation}\label{eqn_f_S_Del}
\begin{split}
f(S_1+a_1,S_2+a_2)-f(S_1,S_2)-b
=& \ h(S+\Delta)-h(S)+(x_r+a_{1r})(x_{2r}+a_{2r})-x_rx_{2r}-b\\
=& \ h(S+\Delta)-h(S)+a_{1r}x_{2r}+a_{2r}x_r+a_{1r}a_{2r}-b,
\end{split}
\end{equation}
where
\begin{equation*}
h(S) = h(S_1,S_2) = h(x_1,x_2,\ldots,x_{2r}) \triangleq \begin{cases}\sum_{1\leq i\leq r-1}x_ix_{i+r}, &  r \ge 2 ,\\
0, & r=1,\\
\end{cases}
\end{equation*}
$\Delta=(a_1,a_2) = (a_{11},a_{12},\ldots,a_{1r},a_{21},a_{22},\ldots a_{2r})$, and $a_2=(a_{21},a_{22},\ldots,a_{2r})$.
For any given $(x_1,x_2,\ldots,x_{2r-1}) \in \F^{2r-1}_{q}$, $a_1\in \F_{q^r}\setminus \{0\}$, and $a_2\in \F_{q^r}$,
the fact $h(S+\Delta)-h(S)+a_{1r}x_{2r}+a_{2r}x_r+a_{1r}a_{2r}-b=0$ has at most one solution $x_{2r}\in \F_q$
implies that \eqref{eqn_f_S_Del} has at most $q^{r-1}$ solutions for all possible $S_2\in \F_{q^r}$, i.e.,
\eqref{eqn_MC} holds. Then $\rho=P_f=\frac{1}{q}$ by Theorem \ref{thm_nonlinear}.

The second assertion is obvious from the definitions.
\end{IEEEproof}

\begin{remark}
(1) For more general form of functions with perfect nonlinearity, similar to $f$ in \eqref{eqn_MM},
the interested reader is referred to \cite{CD,LTH,JZHL,M2016}.
\end{remark}

Recalling the well-known Dillon's class of functions with perfect nonlinearity in Lemma \ref{lemma_Dillon},
we have the following corollary. Note that the trace function $\tr^{q^r}_{q}: \ \F_{q^r} \rightarrow \F_q$
defined by $\tr^{q^r}_{q}(x)=\sum_{0\leq i\leq r-1}x^{q^{i}}$ is a balanced function.

\begin{corollary}
\label{cor_strong_Dillon}
Let $A_1=A_2=\F_{q^{r}}$ and $B=\F_q$, where we regard an element of $\F_{q^{r}}$ as a vector in $\F^{r}_q$.
Let $\{\alpha_1,\alpha_2,\ldots,\alpha_r\}$ and $\{\beta_1,\beta_2,\ldots,\beta_r\}$
be a pair of dual bases of $\F_{q^{r}}$ over $\F_q$, that is,
\begin{equation}\label{eqn_dual_base}
\tr^{q^r}_{q}(\alpha_i\beta_j)=
\begin{cases}
1, &\,\,i=j,\\
0, &\text{otherwise}.
\end{cases}
\end{equation}
Define $f: (\F_{q^{2r}},+) \rightarrow (\F_q,+)$ as $f(x,y)=\tr^{q^r}_{q}(\hat{x}^{q^r-2}\hat{y})$,
where $x=(x_1,x_2,\ldots,x_r) \in \F_{q^r}$, $y=(y_1,y_2,\ldots,y_{r})\in \F_{q^r}$,
$\hat{x}=\sum_{1\leq i\leq r}x_i\alpha_i$, and $\hat{y}=\sum_{1\leq i\leq r}y_i\beta_i$.
Define the probabilisitic encoding map $E_f$ from $A_1$ to $G=A_1\times A_2\times B$ as
\begin{equation*}
E_f(S_1) = (S_1,S_2,f(S_1,S_2)),
\end{equation*}
where $S_1=(x_1,x_2,\ldots,x_{r})\in A_1$, $S_2=(y_1,y_2,\ldots,y_{r})\in_R A_2$.
Then, for the equiprobable encoding case, the systematic strong AMD code given
by $E_f$ has parameters $(q^{r},q^{2r+1},\frac{1}{q})$,
where $\rho=\frac{1}{q}$ is minimum with respect to Theorem \ref{thm_nonlinear}.
Especially, when $r=1$, the $q$-regular AMD code is $G$-optimal
with respect to the bound in Lemma \ref{eqn_bound_m1021}.
\end{corollary}

\begin{IEEEproof}
To prove $\rho=\frac{1}{q}$, according to Theorem \ref{thm_nonlinear}, it suffices to
prove that for any $S_1\in \F_{q^r}$, $a_1 = (a_{11}, \ldots,a_{1r}) \in \F_{q^{r}} \setminus \{(0,0,\ldots,0)\}$,
$a_2 = (a_{21}, \ldots, a_{2r}) \in \F_{q^r}$, and $b \in \F_q$,
\begin{equation*}
\frac{|\{S_2\in \F_{q^r}: \ f(S_1+a_1,S_2+a_2)=f(S_1,S_2)+b\}|}{q^r}\leq \frac{1}{q},
\end{equation*}
i.e., $f(S_1+a_1,S_2+a_2)=f(S_1,S_2)+b$ has at most $q^{r-1}$ solutions for $S_2\in \F_{q^r}$.
Let
\begin{equation}\label{eqn_x'}
\hat{S_1}^{q^r-2}=\sum_{1\leq i\leq r}x'_{1i}\alpha_i,
\end{equation}
\begin{equation}\label{eqn_x^*}
(\hat{S_1}+\sum_{1\leq i\leq r}a_{1i}\alpha_i)^{q^{r}-2}=\sum_{1\leq i\leq r}x^*_{1i}\alpha_i,
\end{equation}
and
\begin{equation}\label{eqn_a'}
a'=(a'_{11}=x^*_{11}-x'_{11},\ldots,a'_{1r}=x^*_{1r}-x'_{1r}).
\end{equation}
Since $a_1\ne (0,0,\ldots,0)$ and $x^{q^r-2}$ is a non-identity permutation of $\F_{q^r}$,
we have $a'\ne (0,0,\ldots,0)$.
Without loss of generality, we may assume $a'_{11}\ne 0$. By \eqref{eqn_dual_base}-\eqref{eqn_a'},
\begin{equation*}
\begin{split}
&\ f(S_1+a_1,S_2+a_2)-f(S_1,S_2)-b\\
=& \ y_1\left(\tr^{q^r}_{q}\left(\beta_1\left(\hat{S_1}+\sum_{1\leq i\leq r}a_{1i}\alpha_i\right)^{q^r-2}\right)-\tr^{q^r}_{q}\left(\beta_1\hat{S_1}^{q^r-2}\right)\right)+C(S,a_1,a_2,b)\\
=& \ a'_{11}y_1+C(S,a_1,a_2,b),\\
\end{split}
\end{equation*}
where $S=(x_1,x_2,\ldots,x_r,y_2,\ldots,y_r) \in \F_{q^{2r-1}}$ and $C(S,a_1,a_2,b)$ is a constant
determined by $S$, $a_1$, $a_2$, and $b$.
Thus, the fact $a'_{11}\ne 0$ means that $f(S_1+a_1,S_2+a_2)-f(S_1,S_2)-b=0$ has at most
$q^{r-1}$ solutions for $S_2\in \F_{q^r}$, which completes the proof.
\end{IEEEproof}


\section{Highly nonlinear functions from algebraic manipulation detection codes}
\label{sec-HNF}

By Theorems \ref{thm_nonlinear_weak} and \ref{thm_nonlinear}, we can construct systematic AMD codes
from known highly nonlinear functions for both weak and strong attack models.
In this section, we further analyze the relationship between AMD codes and highly nonlinear functions.
Specially, we try to construct highly nonlinear functions from some given systematic AMD codes.
Note that a strong $(m,n,\rho)$-AMD code is always a weak $(m,n,\rho)$-AMD code.
Thus, throughout this section, we only consider the functions derived from weak AMD codes.

Let $A_1$, $A_2$ and $B$ be Abelian groups. For a given systematic AMD code
with probabilisitic encoding map $E: A_1 \rightarrow A_1\times A_2\times B$,
$$E(s_1)=(s_1,s_2,t_{s_1,s_2}), \ \ \ \ s_1\in A_1, \ s_2\in_R A_2,$$
define a function $f_E$ from $A_1\times A_2$ to $B$ as
\begin{equation}\label{eqn_AMD_func}
f_E(s_1,s_2)=t_{s_1,s_2}.
\end{equation}

\begin{theorem}
\label{thm_AMDtofunc}
Let $E: A_1 \rightarrow A_1\times A_2\times B$ be the probabilisitic encoding map of
a systematic regular weak AMD code with parameters $(m,n,\rho)$, where $m=|A_1|$ and $n=|A_1||A_2||B|$.
Then the function $f_E: A_1\times A_2 \rightarrow B$ has nonlinearity
$$P_{f_E}\leq \max\{\{\rho\} \cup \{P_{f_{E,s'}}: \ s' \in A_1\}\},$$ where
$f_{E,s'}(x)\triangleq f_{E}(s',x)$ is a function from $A_2$ to $B$ defined by $f_E$ and $s'\in A_1$,
and $P_{f_{E,s'}}$ denotes the nonlinearity of $f_{E,s'}$.
\end{theorem}

\begin{IEEEproof}
Let ${\rho}_{(\Delta_1,\Delta_2,\Delta_3)}$ denote the probability of successful tampering
$(\Delta_1,\Delta_2,\Delta_3)\in (A_1\setminus \{0\})\times A_2\times B$. Then
\begin{equation*}
\begin{split}
\rho \geq & \max\left\{{\rho}_{(\Delta_1,\Delta_2,\Delta_3)}: \ (\Delta_1,\Delta_2,\Delta_3)\in (A_1\setminus\{0\})\times A_2\times B\right\} \\
= & \max_{\Delta_1 \in A_1\setminus\{0\}}\max_{\Delta_2 \in A_2}\max_{\Delta_3\in B}\left\{\sum_{s'\in A_1}\pr(s_1=s')\sum_{s^*\in A_2}\pr(s_2=s^*)\pr(f_E(s'+\Delta_1,s^*+\Delta_2)=f_E(s',s^*)+\Delta_3)\right\}\\
= & \max_{\Delta_1 \in A_1\setminus\{0\}}\max_{\Delta_2 \in A_2}\max_{\Delta_3\in B}\left\{\sum_{s'\in A_1}\frac{1}{|A_1|}\sum_{s^*\in A_2}\frac{1}{|A_2|}\pr(f_E(s'+\Delta_1,s^*+\Delta_2)=f_E(s',s^*)+\Delta_3)\right\}\\
=& \max_{\Delta_1 \in A_1\setminus\{0\}}\max_{\Delta_2 \in A_2}\max_{\Delta_3\in B} \left\{\frac{|\{(s',s^*)\in A_1\times A_2: \ f_E(s'+\Delta_1,s^*+\Delta_2)=f_E(s',s^*)+\Delta_3\}|}{|A_1||A_2|}\right\}.
\end{split}
\end{equation*}
Meanwhile, for $\Delta_1=0$ and $\Delta_2 \in A_2 \setminus\{0\}$, we define
$$\rho_{(0,\Delta_2,\Delta_3)} \triangleq  \frac{|\{(s',s^*)\in A_1\times A_2: \ f_E(s',s^*+\Delta_2)=f_E(s',s^*)+\Delta_3\}|}{|A_1||A_2|}. $$ Then
\begin{equation*}
\begin{split}
& \max_{\Delta_2 \in A_2\setminus\{0\}}\max_{\Delta_3 \in B}\rho_{(0,\Delta_2,\Delta_3)} \\
= & \max_{\Delta_2 \in A_2\setminus\{0\}}\max_{\Delta_3 \in B} \sum_{s'\in A_1}\frac{|\{(s',s^*)\in A_1\times A_2: \ f_{E,s'}(s^*+\Delta_2)=f_{E,s'}(s^*)+\Delta_3\}|}{|A_1||A_2|}\\
\leq & \max\{P_{f_{E,s'}}: \ s' \in A_1\},
\end{split}
\end{equation*}
where the last inequality comes from the fact that
$$P_{f_{E,s'}}=\max_{\Delta_2\in A_2 \setminus \{0\}}\max_{\Delta_3 \in B}\frac{|\{(s',s^*)\in A_1\times A_2: \ f_{E,s'}(s^*+\Delta_2)=f_{E,s'}(s^*)+\Delta_3\}|}{|A_2|}.$$
Therefore,
\begin{equation*}
\begin{split}
P_{f_E} =& \max_{(\Delta_1,\Delta_2) \in A_1\times A_2 \setminus \{(0,0)\}}\max_{\Delta_3 \in B}\frac{|\{(s',s^*) \in A_1 \times A_2: f(s'+\Delta_1, s^*+\Delta_2) = f(s',s^*)+\Delta_3\}|}{|A_1||A_2|}\\
= & \max\{\max\{\rho_{(\Delta_1,\Delta_2,\Delta_3)}: \ \Delta_1 \in A_1 \setminus \{0\}, \Delta_2 \in A_2, \Delta_3 \in B\},
\max\{\rho_{(0,\Delta_2,\Delta_3)}: \ \Delta_2 \in A_2 \setminus \{0\}, \Delta_3 \in B\}\} \\
\leq & \max\{\{\rho\} \cup \{P_{f_{E,s'}}: \ s' \in A_1\}\}.
\end{split}
\end{equation*}
\end{IEEEproof}

Generally speaking, from a systematic AMD code we can not determine the nonlinearity of the function $f_E$ directly.
It is also related with the nonlinearity of some functions with restricted input \cite{MZD}.
This is mainly because that in an AMD code, we do not regard the case $\dec(E(s)+\Delta)=s$
as an adversary's successful tampering, as shown in Theorem \ref{thm_AMDtofunc}.
However, for a stronger setting \cite{CFP,KW,W,WK} also named as \textit{stronger AMD code},
the case $\dec(E(s)+\Delta)\ne \bot$ is regarded
as an adversary's successful tampering.
In this setting, we directly have the following result. The proof is similar, so we omit it here.

\begin{theorem}
\label{coro_AMD_NF}
Let $E: A_1 \rightarrow A_1\times A_2\times B$ be the probabilisitic encoding map
of a systematic regular weak AMD code. If $$\pr(\dec(E(s)+\Delta)\ne \bot)\leq \rho,$$
i.e., it forms a stronger AMD code, then the function $f_E: A_1\times A_2\rightarrow B$ defined as (\ref{eqn_AMD_func})
has nonlinearity $P_{f_E}\leq\rho$.
\end{theorem}


As an application of Theorem \ref{thm_AMDtofunc}, we analyse the functions
derived from the systematic $q$-regular strong AMD codes in \cite[Theorem 2]{CDFPW}.

\begin{corollary}
\label{coro_KAMDtoFunc}
Let $q$ be a power of a prime $p$, and $t>0$ be an integer such that $p \nmid (t+2)$.
Let $(E_h, \dec)$ be the known systematic $q$-regular strong AMD codes in \cite[Theorem 2]{CDFPW}
with parameters $(q^{t},q^{t+2},\frac{t+1}{q})$,
where the probabilistic encoding map $E_h:\F_{q^t}\rightarrow \F_{q^{t}}\times\F_q\times\F_q$ is given by
$$E_h(S=(s_1,s_2,\ldots,s_t)) = (S,x,h(S,x))$$ with $x\in_R\F_q$ and
\begin{equation}\label{eqn_def_h}
h(S,x)=x^{t+2}+\sum_{1\leq t\leq t}s_ix^i.
\end{equation}
Then the function $h(S,x)$ can be viewed as a function from $(\F_{q^{t+1}},+)$ to $(\F_q,+)$ with
nonlinearity $P_h\leq \frac{t+1}{q}$, where we regard elements of $\F_{q^{t+1}}$
as vectors in $\F^{t+1}_q$.
\end{corollary}

\begin{IEEEproof}
According to Theorem \ref{thm_AMDtofunc}, it suffices to prove that for any given $S_1\in \F_{q^{t}}$,
$P_{h_{(E,S_i)}}\leq \frac{t+1}{q}$ holds, where $h_{(E,S_1)}(x)=h(S_1,x)$ is a function from $\F_q$ to $\F_q$.
By \eqref{eqn_def_h}, for any given $S_1=(s_1,s_2,\ldots,s_t)\in \F_{q^t}$ and $\Delta\in \F_q \setminus \{0\}$,
we have
$$h_{(E,S_1)}(x+\Delta)-h_{(E,S_1)}(x)=(x+\Delta)^{t+2}-x^{t+2}+\sum_{1\leq i\leq t}s_i((x+\Delta)^t-x^t)=R_{(S_1,\Delta)}(x),$$
where $\deg (R_{(S_1,\Delta)}(x))=t+1$, since $p\nmid (t+2)$.
Thus, for any $S_1\in \F_{q^{t}}$,
$$P_{h_{(E,S_1)}} = \max_{\Delta\in \F_q \setminus \{0\}}\max_{b\in \F_q}\frac{|\{x \in \F_q: R_{(S_1,\Delta)}(x)=b\}|}{q}\leq \frac{t+1}{q},$$ which completes the proof.
\end{IEEEproof}

\begin{remark}
By Corollary \ref{coro_KAMDtoFunc}, we know that the systematic AMD codes in Theorem 2 of \cite{CDFPW}
can also be explained by highly nonlinear functions.
\end{remark}


\section{Concluding remarks}
\label{sec-conclusion}

In this paper, we investigated the relationship between systematic AMD codes and highly nonlinear functions.
Highly nonlinear functions were used to construct systematic AMD codes.
By carefully choosing highly nonlinear functions, optimal systematic AMD codes,
some with asymptotically optimal tag size, can be generated.
Highly nonlinear functions were also constructed via systematic AMD codes.

However, in general, it is still an open problem whether we can find highly nonlinear functions
to generate optimal AMD codes with optimal tag size.
If it is possible, then how to construct such kinds of highly nonlinear functions is
another interesting topic for future research.
According to Theorem \ref{thm_nonlinear}, the probability of successful tampering
for a systematic strong AMD code is lower bounded by the perfect nonlinearity of the corresponding function.
For the general case, how to construct strong AMD codes via highly nonlinear functions achieving this bound
is also widely open.


\end{document}